\documentclass{article}
\usepackage{amsmath,,amssymb,amsfonts}
\usepackage[american]{babel}
\usepackage{cite}
\usepackage{amsthm}
\usepackage{indentfirst}
\usepackage[latin1]{inputenc}
\usepackage{geometry}
\usepackage{amsfonts}

\newcommand{\mat}[3]{#1\!\equiv\!#2\;\textrm{mod}~#3}

\newcommand{\co}[1]{\left[ #1 \right]}
\newcommand{\pa}[1]{\left( #1 \right)}
\newcommand{\N}{\mathbb{N}}
\newcommand{\F}{\mathcal{F}}
\newcommand{\K}{\alpha}
\newcommand{\Vv}{\vartheta}

\newcommand{\Td}{\sim}
\newcommand{\mm}{\cdot}
\newcommand{\Rt}{\Rightarrow}

\newcommand{\re}{\ldots}

\newtheorem{de}{Definition}[section]
\newtheorem{teo}{Theorem}[section]
\newtheorem{lem}{Lemma}[section]
\newtheorem{prop}{Proposition}[section]

\newtheorem{coro}{Corolary}[section]
\newtheorem{ex}{Example}[section]

\begin{document}
\title{The Conjecture of Collatz 
\thanks{Mathematics Subject Classification 
Primary [$16U60$, $20E05$]; Secondary [$16S34$, $20M05$].
\newline
Keywords: 3n+1, $S-orbits$, Collatz number
\newline Research supported by FAPESP(Funda\c c\~ao  de Amparo \`a 
Pesquisa do Estado de S\~ao Paulo), Proc. 2014/06325-1} }
\author{Costa, G. H. S \and Souza Filho, A.C.}
\date{}
\maketitle
\section{Introduction}

The conjecture we will discuss is generally credited to Lothar Collatz. According to \cite{lagov}, this problem was firstly described in a lecture presented by Collatz in the International Congress of Mathematics, Massachusetts, $1950$.

In what follows, our notation is $\N = \{ 1, 2, \re \}$ and $\N \cup \{ 0 \} = \N_0$, where $\N$ is the set of natural numbers.

The conjecture is then formulated below:

Let $f: \N \longrightarrow \N$ be a map, defined by
$$f(x)= \left\{ \begin{array}{rll} \frac{3x + 1}{2} & \hbox{If} & \mat{x}{1}{2}  \\
\frac{x}{2} ~~~~ & \hbox{if} & \mat{x}{0}{2} \\
\end{array}\right.$$ 
If we denote, for $k \in \N$, $f^k$ the composition of the function $f$ $k$ times. Prove that $\exists \lambda \in \N$, $\forall n \in \N$,  such , $f^{\lambda}(n)=1$. \\

For a given natural number $n$ we can, by composing the function decide if the there exists such exponent $\lambda$. So, $f(7)=\frac{3 \mm 7 + 1}{2} = 11 \Rt f^2(7) = f(11) = \frac{3 \mm 11}{2} = 17 \Rt f^3(7) = f(17) = \frac{3 \mm 17 +1}{2} = 26 \Rt f^4(7) = f(26) = \frac{26}{2} = 13 \Rt f^5(7) = f(13)= \frac{3 \mm 13 +1}{2} = 20$, note that $20=5 \mm 2^2$, thus $f^{5+2}(7) = f^2(5 \mm 2^2)=\frac{5 \mm 2^2}{2^2}=5$. Also, it is clearly that $f^4(5)=1$ and then $f^{11}(7)=1$. Nevertheless, we should be careful, since $f^{70}(27)=1$.

It is a simple routine to prove that, for $k,n \in \N$, $f^k(2^k n)=n$.

\begin{prop} \label{co11}
Let $\{ k, p, n \} \subset \N$, 

\begin{description}
   \item[(i)] $f^k(2^k)=1$;
   \item[(ii)] $f^{k+p}(2^p n)=f^k(n)$.
 \end{description}

\end{prop}

It is a temptation to try techniques of induction  on Collatz conjecture. Nevertheless, if we try it, we see that difficult rises in the induction step. We suppose the conjecture is true when $k\leq n-1$, and look forward to prove it for $k=n$. If $n$ is even, then we are done, since $f(n)=\frac{n}{2}$.
But, be $n$ odd, then $f(n)=\frac{3n+1}{2}=n+\frac{n+1}{2}>n$. After some elapse of time, we try to do it, not module $2$, but module $4$. Now, we see that the odd $n$ such that $n \equiv 1 \pmod 4$, works not for $f(n)$, but for $f^2(n)=\frac{3n+1}{4}$, since $4 \mid 3n+1$, since $ \frac{3n+1}{4}=n+\frac{1-n}{4}<n$, $n \neq 1$, then also we are done. But, there is no way to do it when $n \equiv 3 \pmod 4$.

We will see, next that the condition $n \equiv 3 \pmod 4$, on propose of this work, is the main difficult.

 The paper \cite{evt} is very close to this fact and proves that, if $M$ is a natural number, $\lim_{M \rightarrow  \infty}\frac{A(M)}{M}=1$, where $A(M)$ is the number of natural numbers $m<M$, which for some $\lambda$,  $f^{\lambda}(m)<m$.

In this work, we choose an algebraic approach. Until now, the main gain we have is with the language and notation. But ahead, in the next section, we defined an equivalence relation which seems elucidate the problem. In the final section, we present an algebraic structure to the numbers which satisfies the conjecture, called here Collatz numbers. Then we are able to show infinite sets of Collatz numbers.   

\section{The $S$-Compositions}

If $n$ is a Collatz number which the power $k$, $f^k(n)=1$, then $f^{k+1}(n)=2$ and $f^{k+2}(n)=1$. Clearly, $f^{\lambda + 2i}(n) = 1$,  $i \in \N$.
Thus, when $n$ is a Collatz number we define the natural $\lambda$, the Collatz length associated to $n$, as the minimum power of $f$ such that $f^\lambda(n)=1$.

\begin{de} \label{de12}
Let $\F_f(\N,\N)$ be the family of functions in the natural numbers, such that if $\Phi_n \in \F_f(\N,\N)$ and $\lambda \in \N$, then $\Phi_n(\lambda)=f^\lambda(n) \in \N$. We define $C:\N \longrightarrow \F_f(\N,\N)$, such that, $C(n)=\Phi_n \in \F_f(\N,\N)$ and  $\Phi_n(\lambda)=f^{\lambda}(n)$. We say that $C$ is the  Collatz function if, for all natural number $n$, there exists $\lambda \in \N$, and  $\Phi_n(\lambda)=f^{\lambda}(n)=1$. The least $\lambda$ is the Collatz length of $n$, it is denoted by $\Vv(n)$.
\end{de}

 As defined above, the Collatz conjecture claims that $f$ is a Collatz function.

We can calculate the natural $n \in \N$, where $\Vv (n)=2$. \\

$\begin{array}{l | l}  \textrm{either $\mat{n}{0}{2} \Rt f(n) = \frac{n}{2}$ and} & \textrm{or $\mat{n}{1}{2} \Rt f(n) = \frac{3n + 1}{2}$ and} \\
\hline

 \begin{array}{l | l}  \textrm{either $\mat{\frac{n}{2}}{0}{2} $ }  & \textrm{or $ \mat{\frac{n}{2}}{1}{2}$}\\
  
	\textrm{then $f^{2}(n)=\frac{n}{4}$}  & \textrm{then $f^{2}(n) =$}\\
  \textrm{}  & \textrm{$=\frac{3(\frac{n}{2}) + 1}{2} = \frac{3n + 2}{4}$}\\
  \textrm{}  & \textrm{}\\
	\textrm{$\therefore \frac{n}{4} = 1 \Rt n = 4 $}  & \textrm{$\therefore \frac{3n + 2}{4} = 1 \Rt n \notin \N$}\\
  
 \end{array} & \begin{array}{l | l}
 
   \textrm{either $\mat{\frac{3n + 1}{2}}{0}{2}$,}  & \textrm{or $\mat{\frac{3n + 1}{2}}{1}{2}$,}\\  
    \textrm{then $ f^{2}(n)=\frac{3n + 1}{4}$}  & \textrm{then $ f^{2}(n) =$}\\
    \textrm{}  & \textrm{$ = \frac{3 (\frac{3n + 1}{2}) + 1}{2} = \frac{9n + 5}{4}$}\\
		\textrm{}  & \textrm{} \\
		\textrm{$\therefore \frac{3n + 1}{4} = 1 \Rt n = 1$}  & \textrm{$ \therefore \frac{9n + 5}{4} = 1 \Rt n \notin \N$}\\
    \end{array}
\end{array}$\\

Thus, $\Vv (n)=2$, $n \in \{ 1, 4 \}$.

In \cite{lagov}, it is observed the Collatz was interested in graphical representations
of iteration of functions. Next we define a semigroup which contains all powers of the function $f$. 

\begin{de} \label{de21}
Let the relations:

$$0: \N \longrightarrow \N ~~~~~~~~~~~~~~~~~~1: \N \longrightarrow \N$$
$$~~~~~~~~~~~ x \longmapsto \frac{x}{2} ~~~~~~~~~~~~~~~~~~~~~~ x \longmapsto \frac{3x + 1}{2},$$

denoted by $s_i$, where $i \in \N$. A composition  of a number $k \in \N$ of these relations is called $S$-composition which we denote by $s_{k} \circ s_{k-1} \circ \re \circ s_1 $. We define $S$ the set of all $S$-composition of $f$.
\end{de}
 
Also, we have the notations:

\begin{itemize}

\item $s_{k} \circ s_{k-1} \circ \re \circ s_1 (x) = s_k s_{k-1} \re  s_1 (x)$, so compositions of $f$ is a concatenation of the functions $0$ and $1$. For instance, let $s \in S$, defined by $s=1 \circ 0 \circ 0 \circ 1 \circ 1= 10011$. Since $10011(19)=17$, then $f^5(19)=17$ and $10011$ is one of the $32$ possible powers of $f^5$.

\item  Let $s \in S$, we denote $\underbrace{s \circ s \circ \re \circ s}_{ k-times}=s^k$.

Clearly, $f^k=u^{k_1}v^{m_1} \cdots u^{k_l}v^{m_l}$, and $u,v  \in \{0,1\}$, such that $u \neq v$. For instance, above, we saw that the collatz length of $27$ is $70$. In fact, 
$$0^310^41^30^210^3101010^31^40^21^601^201^30^2101^401^301^20101^501^2(27)=1$$

\item If $s^{-1}$ is the inverse relation of $s$, then we denote  $s_{1}^{-1} \circ s_{2}^{-1} \circ \re \circ s_k^{-1}(y)=(s_k s_{k-1} \re s_1)^{-1}(y)= x$ and $\underbrace{s^{-1} \circ s^{-1} \circ \re \circ s^{-1}}_{ k-times}=s^{-k}$. Hence, $10011(19)=17$ and $s^{-1}(17)=19$.

\end{itemize}

\begin{de}
Let  $S$ be the set of all $S$-compositions of $f$. For all $n \in \N$, we define the set $A_n = \{s \in S~|~s(n) \in \N\}$. If  $k \in \N$, $s \in S$ and $s_\lambda$ is the last index of $s$, we say $\lambda$ the length of $s$, denoted by $l(s) = \lambda$. Hence, we define $ \langle f^{\lambda} \rangle =\{ s \in S ~|~ l(s)=\lambda \}$.
\end{de}

\begin{prop}
Let $n$, $\lambda \in \N$, $\vartheta(n)=\lambda$. Then, there exists a unique $s \in  \langle f^{\lambda} \rangle$, such that $s$ is the less $S$-composition of $A_n$ where $s(n)=1$.
\end{prop}

\begin{prop}
The set $A_n = \{s \in S~|~s(n) \in \N\}$ is uniquely determined by the $S$-composition $s \in S$, such that $s(n)=1$.
\end{prop}

\begin{de}
Let $s \in S$, where $s=s_{k} \circ s_{k-1} \circ \re \circ s_1$. We define $$supp(s)=\sum _{i=1} ^{k} s_i,$$ the support of $s$.
\end{de}

\begin{teo}
Let $S$ be the set of all $S$-compositions of $f$. Then $$S=\bigcup _{k \in \N}  \langle f^k \rangle.$$
\end{teo}

Let $\langle f^k \rangle \subset S$ be the set of all $S$-composition of length $K$. Clearly, $2^k$ is a Collatz number, and $f^k=0^k \in \langle f^k \rangle \subset S$. It is natural to question if there exists $n \in \N \setminus \{2^k\}$ such that $n$ is a Collatz number associated to the $f^k \in \langle f^k \rangle$?

\begin{de}
For each $n \in \N$, we define the sequence $(a_k)_{k \in \N_0}$, where $a_0=n$ and $a_k=f^k(n), k \in \N$, by the $S$-orbit of $n$.
\end{de} 

Let $(a_k)_{k \in \N_0}$ be a $S-orbit$ of $n$, such that $a_l=a_q$, $l \neq q$. Suppose $1<l < q$.
 Then $f^l(n)=u^{k_i}v^{m_i} \cdots u^{k_1}v^{m_1}(n)=u^{x_j}v^{y_j} \cdots u^{x_1}v^{y_1}(n)=f^q(n)$, $i \leq j$ and $u,v \in \{0,1\}$.
Since $k_{p}=x_{p}$ and $m_{p}=y_{p}$, for all $p \leq i$, 
clearly $n_1=u^{k_i}v^{m_i} \cdots u^{k_1}v^{m_1}(n)=x^{k_j}y^{m_j} \cdots x^{k_{i+1}}y^{m_{i+1}}  x^{k_i}y^{m_i}(n) \cdots x^{k_1}y^{m_1}(n)=x^{k_j}y^{m_j} \cdots x^{k_{i+1}}y^{m_{i+1}}(n_1)$.
 Hence a $S-orbit$ such that $n_1=f^{q-l}(n_1)$, and $n_1=f^l(n)$ in the sequence $(a_k)_{k \in \N_0}$.

Thus, if $a_l=a_q$, $l \neq q$ then the $S-orbit$ of $n$ repeats one of its elements.

\begin{de}
Let $n \in \N$ be such that the $S-orbit$ $(a_k)_{k \in \N_0}$ has the property that $a_l=a_q$, for some $l<q$.
 Then, the $S-orbit$ of $n$ is cyclic. 
If $a_l \neq a_q$, for all $l,q \in \N$, then the $S-orbit$ of $n$ is non-cyclic.
\end{de}

\begin{prop}
If a $S-orbit$ of one natural $n$ is non-cyclic, then it has infinite elements.
\end{prop}

Clearly, if $n$ is a collatz number of length $\vartheta$, then the $S-orbit$ of $n$ has exactly $\vartheta+1$ distinct elements, althought it is infinite since $(n,f(n), \cdots, f^{\vartheta-1},1,2,1,2,1, \cdots)$ is the $S-orbit$ of $n$. Also, the $S-orbit$ of any Collatz number is cyclic.

\begin{teo}
Let $\Td:\N \longrightarrow \N$ be the relation define as follows: $n\Td m$ if $n, m$ are elements of the same $S-orbit$. Then $\Td$ is an equivalence relation. Futhermore $\Td$ has exactly one class $\overline{1}$, the class of Collatz number if, and only if, Collatz's conjecture is true.
\end{teo} 
%

Let $(a_k)_{k \in \N_0}$ be a cyclic or an infinite $S-orbit$ of $n \in \N$. The set $A=\{a_k,$ the terms of the $S-orbit$ of $n\}$ is the class of the natural numbers $m \Td n$. By the Well Ordering Principle, $A \subset \N$ has a minimum element $w_n \in A$. Since $w_n \Td n$, we denote the class $A$ by $\overline{w_n}$. Since any equivalence class of $\Td$ represents an $S-orbit$ of any $n \in \N$, the set of all equivalence classes of $\Td$ is a totally ordered set with a partition $\overline{1} \cup \mathfrak{C} \cup \mathfrak{F}$, were $\mathfrak{C}$ is the set of all cyclic $S-orbit$ and $\mathfrak{F}$ the class of all infinite $S-orbit$.

\begin{prop}
Let $(a_k)_{k \in \N_0}$ be a cyclic or an infinite $S-orbit$ of $n \in \N$m and $w_n$ the minimum element of $(a_k)_{k \in \N_0}$. Then $w_n \equiv 3 \pmod 4$.
\end{prop}

In \cite{evt}, is proved a very interesting result which allow us to conclude that almost every natural number has the property that $f^l(n)<n$ for some $k$. Since the set $F$ of all infinite $S-orbit$, is an order set, the result of Everett indicates that $F$ has an upper bound?

Let $(a_k)_{k \in \N_0}$ be cyclic. 
Then we can find a solution for the equation $f^k(x)=x$.
 If we iterate $f$, the solution is  of the form $x=\frac{b}{2^k-a}$.
 Thus, for $k$, the set $C^k=\{(a,b)|f^k(x)=frac{ax+b}{2^4}\}$ is of the form $C^{K-1} \cup fC^{k-1}$. 
If $k=3$ this set is
$$\{(1,0),(3,1),3,2)3^2,3+2)\}\cup \{(3,2^2),(3^2,3*2+2^2),(3^2,3+2^2),(3^3,3^2+3*2+2^2)\}.$$
Which the solutions is $x=0$, $x=-10$, $x=-7$, $x=-5$ and $x=-1$. If $k=2$ we obtain the solution $x=1$. If $k>2$, we realize that when $2^k-a>0$, $2^k-a \not{\mid} b$. 

Also, a cyclic $S-orbit$ of n, such that $f^k(n)=n$, induce cyclic $S-orbits$ of $2^m*n$. In fact, $f^m(2^m*n)=n$, and then $f^{m+k}(2^m*n)=f^k(n)=n$.

Clearly, the set $2^N=\{1,2,2^4, \cdots, 2^n, \cdots \}$ is of Collatz numbers. Next we present non-trivial sets likely.

\section{An algebraic structure of Collatz number}

The main result of this section is the theorem:

\begin{teo}[An algebraic structure] \label{te k}
Let  $k_i, m_i \in \N$, $1\leq i \leq n$. If $0^{k_1}1^{m_1} \cdots 0^{k_n}1^{m_n}(\alpha)=1$, then
$$\alpha=\co{\frac{2}{3}}^{m_{n}}\pa{2^{k_{n}}\pa{\co{\frac{2}{3}}^{m_{n-1}}\pa{2^{k_{n-1}} \pa{ \cdots \pa{\co{\frac{2}{3}}^{m_{1}}\pa{2^{k_{1}}+1}-1} \cdots} +1}-1}+1}-1  \ \ \ \ \ \  \ (\star)$$ 
\end{teo}

The following results are used to prove this theorem.

\begin{lem} \label{K}
Let $m$ be a natural  number and $1^{m} \in S$. If $\alpha$ and $\alpha_1$ are natural numbers, and $\alpha_{1}=1^{m}(\alpha)$, then $\alpha=\co{\frac{2}{3}}^m(\alpha_1+1)-1$.
\end{lem}


\begin{lem} \label{KI}
Let $\alpha$, $n \in \N$ and  $0^{k_1}1^{m_1} \cdots 0^{k_n}1^{m_n}(\alpha)=1$. To $1\leq i\leq n$, we recursively calculate $\alpha$ to of the  powers $0^{k_i}1^{m_i}$, with $\alpha_i=2^{k_i}\beta_{i-1}$, $\beta_i=1^{-m_i}(\alpha_i)$ and $\alpha_1 = 0^{-k_1}(1) = 2^{k_1}$. Hence, finishing the iterations when $\alpha = \beta_n$.
\end{lem}


Now we prove the algebraic theorem proposed.

\begin{proof}
We prove by induction on $n$, the index of the pairs of powers.
 If $n=1$, then $0^{k_1}1^{m_1}(\alpha)=1$ and  $0^{-k_1} \circ 0^{k_1}1^{m_1}(\alpha)=1^{m_1}(\alpha)=0^{-k_1}(1) = 2^{k_1}$ and by Lemma  \ref{K}, $\alpha = \co{\frac{2}{3}}^{m_1}(2^{k_{1}}+1)-1$.
 Since $0^{k_1}1^{m_1} \re 0^{k_i}1^{m_i}(\alpha)=1$, we assume $$\alpha=\co{\frac{2}{3}}^{m_{i}}\pa{2^{k_{i}}\pa{\co{\frac{2}{3}}^{m_{i-1}}\pa{2^{k_{i-1}} \pa{ \cdots \pa{\co{\frac{2}{3}}^{m_{1}}\pa{2^{k_{1}}+1}-1} \cdots} +1}-1}+1}-1$$ as the I. H.. 
Thus, to $i+1$, we have 
$0^{k_1}1^{m_1} \re 0^{k_i}1^{m_i} 0^{k_{i+1}} 1^{m_{i+1}} (\alpha)=1$
 $(\star)$ and $\alpha ' = 0^{k_{n+1}}1^{m_{n+1}}(\alpha)$.
 In this way, we rewrite the expression $(\star)$ as $0^{k_{1}}1^{m_{1}} \re 0^{k_{i}}1^{m_{i}}(\alpha ')=1$. 
By I.H., $$ \alpha ' =\co{\frac{2}{3}}^{m_{i}}\pa{2^{k_{i}}\pa{\co{\frac{2}{3}}^{m_{i-1}}\pa{2^{k_{i-1}} \pa{ \cdots \pa{\co{\frac{2}{3}}^{m_{1}}\pa{2^{k_{1}}+1}-1} \cdots} +1}-1}+1}-1.$$
 Hence again, by Lemmas \ref{K} as \ref{KI} we have $\alpha= \co{\frac{2}{3}}^{m_{i+1}} (2^{k_{i+1}} \alpha ' +1) -1$ $(\star \star)$, substituting $\alpha '$ in $(\star \star)$, we conclude $$ \alpha =\co{\frac{2}{3}}^{m_{i+1}}\pa{2^{k_{i+1}}\pa{\co{\frac{2}{3}}^{m_{i}}\pa{2^{k_{i}} \pa{ \cdots \pa{\co{\frac{2}{3}}^{m_{1}}\pa{2^{k_{1}}+1}-1} \cdots} +1}-1}+1}-1.$$

\end{proof}

Writing $0^{k_1}1^{m_1} \cdots 0^{k_n}1^{m_n}(\alpha)=1$, clearly, if $\alpha \in \N$, it is a Collatz number of length $\vartheta(\alpha) = \sum_{i=1}^{n}(k_i+m_i)$. Hence, the summ of the powers is a  partition of $\vartheta (\alpha)$. So, if we want to determine the Collatz numbers of a predefined length $\vartheta$, we represent it with the properly partition of $\vartheta$. Bellow, we list simple rules which reduces considerably the amount of this partitions.

Let $\vartheta$ be natural and $P=\{1+1+ \cdots+1, 2+1+ \cdots+1, \cdots, 1+(\vartheta-1),\vartheta\}$ the set of all partitions of $n$. If $0^{k_1}1^{m_1} \cdots 0^{k_n}1^{m_n}(\alpha)=1$ and $\vartheta(\alpha) = \sum_{i=1}^{n}(k_i+m_i)$, we can reduce the elements of $p \in P$ such that $\alpha$ is a Collatz number. 

Although, $p \in P$ is such that $p=p_1+p_2+\cdots +p_n=p_2+p_1+\cdots +p_n$, the powers $k_1=p_1, k_2=p_2, \cdots k_n=p_n$ and $k_1=p_2, k_2=p_1, \cdots k_n=p_n$, if define a Collatz number $\alpha$, they produce different numbers. Any away, we can use the general rules bellow:
\begin{enumerate}
\item $p$ cannot start with $1$.
\item $p$ cannot start with $01$
\item $k_1 \equiv 1 \pmod 2$, if $\alpha \neq 2^\vartheta$
\end{enumerate}

In fact, if we want to determine the numbers with Collatz length $6$, we write\\
$P=\{1+1+1+1+1+1, \\ 
2+1+1+1+1, \\
2+2+1+1, 3+1+1+1, \\
2+2+2,3+2+1,4+1+1, \\
2+4,3+3,5+1, \\
6\}$.

As defined before, the Collatz of length $6$ are elements of the set $f^6=\{000000,000001, \cdots 111111\}$ which $|f^6|=2^6$. The rule $1.$ reduces this set to $2^5$. The rule $2.$ reduces $2^5$ to $2^5-2^4=2^4$ and the rule $2.$ reduces it to $6$. Also, the rules exclude the partitions: \\
with $6$ summands: $010101$, \\
with $5$ summands: $1^2(01)^2,10^2101,101^201,1010^21,(10)^21^2$; $01^2010,010^210,0101^20,(01)^20^2$ and $0^21010$;\\
We going on listing the no excluding partions:\\
with  $4$ summands: $0^3101$\\
with $3$ summands:$0^31^20,0^310^2$\\
with $2$ summands:$0^31^3$ and $0^51$\\
with $1$ summands: it always exists $0^6$.\\

And the Collatz numbers of length $6$ are $2^6,\ 21, \ 6$ and $20$. Hence    $0^6(2^6)=0^51(21)=0^31^20(6)=0^310^2(20)=1$

\begin{coro}
If $0^{k}1^{m}(\alpha)=1$, then $\alpha=\co{\frac{2}{3}}^m(2^k+1)-1$ and $\alpha \in \N \Longleftrightarrow k=(2l-1)3^{m-1}$, $l \in \N$.
\end{coro}
\begin{proof}
By the Theorem, $\alpha=\co{\frac{2}{3}}^m(2^k+1)-1$. If $\alpha \in \N$, then $2^m(2^k+1) \equiv 0 \pmod {3^m}$. Since $(2.3)=1$, the CGD, then 
$2^k \equiv -1 \pmod {3^m}$, in particular $3 \mid 2^k-1$, then $k$ is an odd number. Clearly $2^{2k} \equiv 1 \pmod {3^m}$. If $\phi(3^m)$ denotes the number of positive integers which are prime relative to $3^m$, then $\phi(3^m)=2*3^{m-1}$. By the Theorem of Euler, $2k \equiv 0 \pmod {2*3^{m-1}}$, then $k=q*3^{m-1}$. Since $k$ is odd we have $k=(2l-1)3^{m-1}$, $l \in \N$.
\end{proof}

According to the corollary, the set $\{\co{\frac{2}{3}}^m(2^k+1)-1, k=(2l-1)3^{m-1}, l,m \in \N\}$ is an infinite subset of the natural numbers which every element is a Collatz number. In \cite{and} some sets are also presented, but in a different way.

Next we present some sets obtained as solutions of the algebraic structure present in the last theorem.

\section{Sets of Collatz numbers}

\begin{teo}
Let $0^{k_1}1^{m_1}0^{k_2}1^{m_2}(\alpha)=1$ ($\star$). Then $\alpha=\co{\frac{2}{3}}^{m_{2}}\pa{2^{k_{2}} \pa{\co{\frac{2}{3}}^{m_{1}}\pa{2^{k_{1}}+1}-1} +1}-1$. Futhermore,  the set $$\{\alpha, \ k_2=2l_2*3^{ m_{2}-1} \ and \ k_1=(2l_1-1)3^{m_1+m_2-1}, \ l1,l_2,m_1,m_2 \in \N\}$$ is a
 subset of the set of the solutions of $\alpha \in \N$.
\end{teo}

\begin{proof}
The equation $\star$ shows that $\alpha=2^{m_2} \frac{2^{k_{2}+m_1} \pa {2^{k_{1}}+1} - 3^{m_1} \pa{2^{k_{2}}-1}}{3^{m_{1} + m_{2}}}-1$.\\ 
Then 
$2^{k_{2}+m_1} \pa {2^{k_{1}}+1} \equiv 3^{m_1} \pa{2^{k_{2}}-1} \pmod {3^{m_{1} + m_{2}}}$. We present a particular case where we can obtain a solution. Suppose  $2^{k_{2}}-1 \equiv 0 \pmod {3^{m_{2}}}$, then $2^{k_{1}}+1 \equiv 0 \pmod {3^{m_{1}+m_{2}}}$. 
Now, the solutions are trivial so that $k_2$ is an even number and $k2=2l_2*3^{ m_{2}-1}$ and $k_1$ is even and $k_1=(2l_1-1)*3^{m_1+m_2-1}$, and we determine a subset of the set of the solutions.
The set $\{\alpha, \ k_2=2l_2*3^{ m_{2}-1} \  and \ k_1=(2l_1-1)3^{m_1+m_2-1}, \ l1,l_2,m_1,m_2 \in \N\}$ is a subset of the set of the solutions.
\end{proof}

It is clear, by divisibility properties, $2^{k_{2}+m_1} \pa {2^{k_{1}}+1} \equiv 0 \pmod {3^{m_1}}$, and since $(3,2)=1$, we have $ 2^{k_{1}}+1 \equiv 0 \pmod {3^{m_1}}$, which the solution is in the corollary of the theorem. Thus $2^{k_{1}}+1=q*3^{m_1}$, and $2^{k_{2}+m_1} q*3^{m_1} \equiv 3^{m_1} \pa{2^{k_{2}}-1} \pmod {3^{m_{1} + m_{2}}}$. Then we have 
$2^{k_{2}+m_1}*q \equiv \pa{2^{k_{2}}-1} \pmod {3^{ m_{2}}}$. Then we can consider the conditions $(q,3)=1$ or $(q,3^k_q)=3^k_q$. The solutions can be obtained, but they are not easily described. In fact, $0^310^41^4(15)=1$, thus $k_1=3$, $m_1=1$, $k_2=m_2=4$ is a solution to $\alpha$, but not in the particular condition  $2^{k_{2}}-1 \equiv 0 \pmod {3^{m_{2}}}$. If we choose, for the subset of the theorem before $m_2=4$, then the minor $k_2=54$ and $m_1=1$, $k_1=81$ and $\alpha >>15$.

We can proceed as in the last theorem and determine subsets of the solutions when $$\alpha=\co{\frac{2}{3}}^{m_{3}}\pa{2^{k_{3}}\pa{ \co{\frac{2}{3}}^{m_{2}}\pa{2^{k_{2}} \pa{\co{\frac{2}{3}}^{m_{1}}\pa{2^{k_{1}}+1}-1} +1}-1} +1}-1$$

\begin{teo}
Let $0^{k_1}1^{m_1}0^{k_2}1^{m_2}0^{k_3}1^{m_3}(\alpha)=1$ ($\star$).\\ Then 
$\alpha=\co{\frac{2}{3}}^{m_{3}}\pa{2^{k_{3}}\pa{ \co{\frac{2}{3}}^{m_{2}}\pa{2^{k_{2}} \pa{\co{\frac{2}{3}}^{m_{1}}\pa{2^{k_{1}}+1}-1} +1}-1} +1}-1$. 
Futhermore,  the set
$\{\alpha, \textrm{$k_3=2l_3*3^{m_3-1}$, $k_2=2l_2*2^{m_2+m_3-1}$ and $k_1=(2l_1-1)*3^{m_1+m_2+m_3-1}$, $l_i, m_i \in \N, 1\leq i \leq 3$}\}.$
is a
 subset of the set of the solutions of $\alpha \in \N$.

\end{teo}

\begin{proof}
 Write $\alpha=\frac{2^{m_3}}{3^{m_{3}+m_{2}+m_{1}}}\pa{2^{k_{3}}\pa{ 2^{m_2+k_{2}+m_1} \pa {2^{k_{1}}+1}-3^{m_1}\pa{2^{k_{2}}-1}}+3^{m_{2}+m_{1}}(1-2^{k_{3}})    }-1$. 
If $\alpha \in \N$, then $2^{k_3+m_2+k_{2}+m_1} \pa {2^{k_{1}}+1} 
\equiv 3^{m_1}\pa{2^{k_3}(2^{k_{2}}-1)-3^{m_{2}}(1-2^{k_{3}})} \pmod{3^{m_{3}+m_{2}+m_{1}}}$.
 We suppose $2^{k_{1}}+1 \equiv 0 \pmod{3^{m_{3}+m_{2}+m_{1}}}$, then $2^{k_3}(2^{k_{2}}-1) \equiv 3^{m_{2}}(1-2^{k_{3}}) \pmod{3^{m_{3}+m_{2}}}$. Now, we take $2^{k_{2}}-1\equiv 0 \pmod{3^{m_{3}+m_{2}}}$, hence $1-2^{k_{3}} \equiv 0 \mod \pmod{3^{m_{3}}}$, thus $k_3$  and $k_2$ are even number. And 
$$\{\alpha, \textrm{$k_3=2l_3*3^{m_3-1}$, $k_2=2l_2*2^{m_2+m_3-1}$ and $k_1=(2l_1-1)*3^{m_1+m_2+m_3-1}$, $l_i, m_i \in \N, 1\leq i \leq 3$}\}.$$
\end{proof}

If $m_1=m_2=m_3=1$, $l_1=l_2=l_3=1$ then $k_3=2$, $k_2=6$ and $k_1=9$, then $\alpha=38797$ and $f^{20}(\alpha)=0^910^610^21(38797)=1$.

But these solutions increases exponentially, in fact
If $m_1=m_2=1$ and $m_3=2$, $l_1=l_2=l_3=1$ then $k_3=6$, $k_2=18$ and $k_1=27$, then $\alpha=444799961540067$ and has length $55$

\begin{teo}
Let  $k_i, m_i \in \N$, $1\leq i \leq n$. If $0^{k_1}1^{m_1} \cdots 0^{k_n}1^{m_n}(\alpha)=1$, then there exists $\alpha \in \N$. In particular $$\{\alpha, \textrm{$k_n=2l_n*3^{m_n-1}$, $k_{n-1}=2l_{n-1}*3^{m_{n-1}+m_n-1}, \cdots,$$
$$ \ k_2=2l_2*3^{m_2+ \cdots +m_n-1}$}$$
$$\textrm{and $k_1=(2l_1-1)*3^{m_1+\cdots+m_n-1}$, $l_i, m_i \in \N, 1\leq i \leq n$}\}$$ is a subset of the set of solutions for $\alpha$.
\end{teo}

\bibliographystyle{amsalpha}

\vspace{.5cm}
\begin{thebibliography}{A}
\bibitem{and}ANDREI, S.; KUDLEK M.; NICULESCU, R. S.. \textbf{Some results on the Collatz problem}. Acta Informatica, V. 37, 2000. p. 145-160.\\


\bibitem{evt}Everett, C. J., {\it Iteration of the number-theoretic function $f(2n)=n$, $f(2n+1)=3n+2$}, Advancaes in Mathematics $25$, $42-45$, $(1977)$.
\bibitem{lagov} Lagarias, J. C., {\it The 3x+1 problem: an overview.  The ultimate challenge: the 3x+1 problem},  3-29, Amer. Math. Soc., Providence, RI, 2010. 


\end{thebibliography}
\par
\vspace{5mm}

\noindent Juriaans, S.O.,\newline
Instituto de Matemática e Estatística \newline
Universidade de São Paulo \newline
CP 66281 \newline
CEP 05311-970 \newline
São Paulo - Brazil\newline
ostanley@usp.br

\vspace{5mm}

\noindent De A. E Silva, A., \newline
Departamento de Matem\'atica \newline
Universidade Federal da Para\'iba \newline
CEP 58051-900, \newline
Jo\~ao Pessoa, Pb, Brazil, \newline
andrade@mat.ufpb.br

\vspace{5mm}

\noindent Souza Filho, A. C., \newline
Escola de Artes, Ciências e Humanidade \newline
Universidade de São Paulo \newline
Rua Arlindo Béttio, 1000 \newline
CEP 03828-000 \newline
São Paulo - Brazil \newline
acsouzafilho@usp.br

\end{document}